\documentclass[12pt,reqno]{article}
\usepackage[usenames]{color}
\usepackage{amssymb}
\usepackage{graphicx}
\usepackage{verbatim}
\usepackage[colorlinks=true,
linkcolor=webgreen,
filecolor=webbrown,
citecolor=webgreen]{hyperref}
\definecolor{webgreen}{rgb}{0,.5,0}
\definecolor{webbrown}{rgb}{.6,0,0}
\usepackage{color}
\usepackage{fullpage}
\usepackage{float}
\usepackage{enumerate}
\usepackage{amsmath}
\usepackage{amsthm}
\usepackage{amsfonts}
\usepackage{latexsym}
\usepackage{epsf}
\newcommand{\seqnum}[1]{\href{http://oeis.org/#1}{\underline{#1}}}

\begin{document}
\begin{center}
\epsfxsize=4in
\end{center}

\theoremstyle{plain}
\newtheorem{theorem}{Theorem}
\newtheorem{corollary}[theorem]{Corollary}
\newtheorem{lemma}[theorem]{Lemma}
\newtheorem{proposition}[theorem]{Proposition}
\theoremstyle{definition}
\newtheorem{definition}[theorem]{Definition}
\newtheorem{problem}[theorem]{Problem}
\newtheorem{example}[theorem]{Example}
\newtheorem{conjecture}[theorem]{Conjecture}
\newtheorem{claim}[theorem]{Claim}
\newtheorem{remark}[theorem]{Remark}
\allowdisplaybreaks

\begin{center}
\vskip 1cm{\LARGE\bf On Polynomial Pairs of Integers}\
\vskip 1cm
\large Martianus Frederic Ezerman\\
Division of Mathematical Sciences\\
School of Physical and Mathematical Sciences\\
Nanyang Technological University\\
21 Nanyang Link, Singapore 637371\\
\href{mailto:fredezerman@ntu.edu.sg}{\tt fredezerman@ntu.edu.sg}\\
\ \\
Bertrand Meyer\\
Telecom ParisTech\\
46 rue Barrault, 75634 Paris Cedex 13, France\\
\href{mailto:meyer@enst.fr}{\tt meyer@enst.fr}\\
\ \\
Patrick Sol\'{e}\\
Telecom ParisTech\\
46 rue Barrault, 75634 Paris Cedex 13, France\\
and\\
Mathematics Department, King Abdulaziz University\\
P.~O.~Box 80203, Jeddah 21589, Saudi Arabia\\
\href{mailto:sole@enst.fr}{\tt sole@enst.fr}\\
\end{center}

\def\Q{{\mathbb Q}}
\def\F{{\mathbb F}}
\def\N{{\mathbb N}}
\def\DD{\mathcal{D}}
\def\SS{\mathcal{S}}
\def\supp{\operatorname{supp}}
\def\wt{\operatorname{wt}}
\newcommand{\gen}[1]{{\langle #1 \rangle}}

\vskip .2 in
\begin{abstract}
The reversal of a positive integer $A$ is the number obtained by reading $A$ backwards in its decimal representation. 
A pair $(A,B)$ of positive integers is said to be palindromic if the reversal of the product $A \times B$  is equal to the 
product of the reversals of $A$ and of $B$. A pair $(A,B)$ of positive integers is said to be polynomial if the product 
$A \times B$ can be performed  without carry. 

In this paper, we use polynomial pairs in constructing and in studying the properties of palindromic pairs.
It is shown that polynomial pairs are always palindromic. It is further conjectured that, provided that neither $A$ nor $B$ 
is itself a palindrome, all palindromic pairs are polynomial. A connection is made with classical topics in recreational 
mathematics such as reversal multiplication, palindromic squares, and repunits.
\end{abstract}
\section{Introduction}
On March 13, 2012 the following identity appeared on K.~T.~Arasu's Facebook posting\footnote{His account has since been deactivated 
for personal reasons.}:
\begin{quote}
Notice that
\begin{equation*}
25986 = 213 \times 122\text{.}
\end{equation*}
Now, read the expression above in reverse order and observe that
\begin{equation*}
221 \times 312 = 68952\text{.}
\end{equation*}
\end{quote}

The point here is, of course, that the second equality still holds in first order arithmetic!

For $A \in \N$ let the {\it reversal} of $A$, denoted by $A^{*}$, be the integer obtained by reading $A$ backwards in base $10$. 
We say that $A$ is a {\it palindrome} if $A=A^{*}$. In this paper, we investigate how to determine which pairs $(A,B)$ of positive 
integers satisfy the property
\begin{equation}\label{eq:one}
C = A \times B \text{ and } C^{*} = A^{*} \times B^{*} \text{.}
\end{equation}
In words, the product of the reversals is the reversal of the product. We shall call such a pair a {\it palindromic pair}.

Note that there are integers $C$ with more than one corresponding pair $(A,B)$ satisfying Equation~\ref{eq:one}. For example, we have
\begin{equation*}
2448 = 12 \times 204 = 24 \times  102 \text{.}
\end{equation*}
Upon reversal, one has
\begin{equation*}
8442 = 21 \times 402 = 42 \times  201 \text{.}
\end{equation*} 
It is easy to see that palindromic pairs always occur in distinct 
pairs $(A,B)$ and $(A^{*},B^{*})$ unless both $A$ and $B$ are palindromes. The pair $(12,13)$, for instance, comes with the pair 
$(21,31)$ upon reversal.

The question of how to characterize palindromic pairs had appeared in~\cite[p.\ 14]{BC} where the pair $(122,213)$ was given, yet this matter 
has hardly been looked at more closely. In this short note we introduce the notion of {\it polynomial pairs} as a tool to study palindromic pairs. 
We show that all the examples of palindromic pairs presented below can be explained in terms of polynomial pairs. We conjecture, but cannot yet 
prove that, when neither $A$ nor $B$ is a palindrome, all palindromic pairs $(A,B)$ are polynomial pairs.

The concept of polynomial pairs intersects many classical topics in recreational mathematics. An interesting topic concerns the
{\it repunits} which are numbers all of whose digits are $1$~\cite[Ch.\ 11]{B}. Another topic deals with a known technique to produce 
palindromes usually referred to as {\it reversal multiplication}~\cite{won}. The integers $A$ with the property that  $A \times A^{*}$ is 
a palindrome form Sequence $A062936$ in the Online Encyclopedia of Integer Sequences (henceforth, OEIS)~\cite{OEIS}. 
Whenever $(A,A^{*})$ form a polynomial pair, we learn that reversal multiplication always produces a palindrome. 

The material is arranged as follows. Section~\ref{sec:formal} introduces and Section~\ref{sec:three} develops the concept of palindromic pairs. 
Section~\ref{sec:four} explores the multiplicity of the representation of a repunit as a product of a palindromic pair of integers. Section~\ref{sec:five} covers the special case where the two members of a palindromic pair are reversals of each other. Section~\ref{sec:six} is dedicated to palindromes that are perfect squares. Section~\ref{sec:seven} replaces multiplication by addition in the definition of a palindromic pair. The paper ends with a summary.

\section{Formalization}\label{sec:formal}
To formalize the problem mathematically, some notation is in order. Let 
\begin{equation*}
A := \sum_{i=0}^{a} a_{i} 10^{i}
\end{equation*} 
be an integer expressed in base $10$. 
Its reversal is
\begin{equation*}
A^{*}:=\sum_{i=0}^{a} a_{i} 10^{a-i}\text{.}
\end{equation*} 
Define the polynomial $P(A,x):=\sum_{i=0}^{a} a_{i} x^{i} \in \Q[x]$
so that $P(A,10)=A$. 
Then its {\it reciprocal}
\begin{equation*}
P^{*}(A,x):=\sum_{i=0}^{a} a_{i} x^{a-i}=x^{a} P(A,1/x)
\end{equation*}
satisfies $P^{*}(A,10)=A^{*}$. A pair $(A,B)$ of not necessarily distinct positive integers is said to be a {\it palindromic pair} if
\begin{equation*}
P^*(A,10) P^*(B,10) = P^*(A \times B,10) \text{.}
\end{equation*}

We shall say that the pair $(A,B)$ is {\it polynomial} if 
\begin{equation*}
P(A,x)P(B,x) = P(A \times B,x) \text{.}
\end{equation*}

The pair $(12,21)$, for instance, is a polynomial pair since
\begin{equation*}
P(12 \times 21,x)=2x^{2}+5x+2 = (x+2)(2x+1)\text{,}
\end{equation*}
but $(13,15)$ is not a polynomial pair because
\begin{equation*}
(x+3) (x+5) = x^{2} + 8x + 15 \text{ while } P(13 \times 15,x) = x^{2}+9x+5 \text{.}
\end{equation*}

The following characterization of polynomial pairs will be used repeatedly.
\begin{proposition}\label{prop:equiv}
The following assertions are equivalent.
\begin{enumerate}[1)]
\item $(A,B)$ is a polynomial pair.
\item The multiplication of $A$ by $B$ can be performed without carry.
\item The coefficients of the polynomial $P(A,x)P(B,x)$ are bounded above by $9$.
\end{enumerate}
\end{proposition}
\begin{proof} 
Let $j$ be the smallest integer such that $c_{j}:=\sum_{j=i+k} a_{i} b_{k} > 9$. Then $c_{j}$ is 
the coefficient of $x^{j}$ in $P(A,x) P(B,x)$, while the coefficient of $x^{j}$ in $P(A \times B,x)$ 
is $c_{j} \pmod{10} \neq c_{j}$. Thus, 1) implies 2) by contrapositive argument.

Now, assume that there is some $j$ such that $c_{j}:=\sum_{j=i+k} a_{i} b_{k} > 9$. 
Then, in the multiplication $A \times B$, the term $y:= \left \lfloor \frac{c_{j}}{10} \right \rfloor$ 
is carried over to the coefficient of $10^{j+1}$. This establishes that 2) implies 3).

Lastly, to show that 3) implies 1), we begin by substituting $x=10$. Hence,
\begin{equation*}
A \times B = P(A,10)P(B,10)=\sum_{k=0}^{a+b} \left(\sum_{i+j=k} a_{i} b_{j}\right) 10^{k}.
\end{equation*}
The coefficient of $x^{k}$ in $P(A \times B, x)$ is $\sum_{k=i+j} a_{i} b_{j}$, which is assumed to be $\leq 9$. 
This means that $(A,B)$ is indeed a polynomial pair. The proof is therefore complete.
\end{proof}

Polynomial pairs are palindromic pairs as the next result shows.
\begin{proposition}
A polynomial pair $(A,B)$ is palindromic.
\end{proposition}
\begin{proof}
If $P(A,x)P(B,x)=P(A \times B,x)$, then, by taking reciprocals, we get
\begin{equation*}
P^*(A,x)P^*(B,x)=P^*(A \times B,x)\text{.}
\end{equation*}
Using $x=10$ completes the proof.
\end{proof}

This observation raises an initial question:
\begin{problem}\label{prob:1}
Are there palindromic pairs that are not polynomial?
\end{problem}

Our investigation quickly reveals that the answer is yes. If we allow either $A$ or $B$ to be palindromes then there are palindromic pairs which are not polynomial pairs.

The test that a pair $(A,B)$ is palindromic is done simply by checking if the definition is satisfied. We record $A,B$, and $C$ whenever we have $A \times B = C$ and $A^{*} \times B^{*} = C^{*}$. We then perform a check if the multiplication of $A$ by $B$ can be performed without carry. The pair $(A,B)$ that fails to pass this check is not a polynomial pair by Proposition~\ref{prop:equiv}. Table~\ref{table:NonPol} provides the list of all such pairs $(A,B)$ with $A \leq B$ and $A \times B = C \leq 10^{7}$ generated by exhaustive search. We requires $A \leq B$ to avoid duplication of pairing.

\begin{table}
\caption{Palindromic but not polynomial pairs $(A,B)$ with $A \leq B$ and $A \times B \leq 10^{7}$}
\label{table:NonPol}
\centering
\begin{tabular}{ c  c  | c   c  | c   c  | c  c }
\hline
$(A,B)$ & $A \times B$ & $(A,B)$ & $A \times B$ & $(A,B)$ & $A \times B$ & $(A,B)$ & $A \times B$ \\
\hline
$ ( 7 , 88) $  & $ 616 $ & $ ( 555 , 979) $ & $ 543345 $   & $ ( 737 , 888) $  & $ 654456 $  & $ ( 707 , 8558) $ & $ 6050506 $ \\
$ ( 8 , 77) $  &         & $ ( 55 , 9999) $ & $ 549945 $   & $ ( 777 , 858) $  & $ 666666 $  & $ ( 7 , 880088) $ & $ 6160616 $ \\
$ ( 55 , 99) $ & $5445 $ & $ ( 99 , 5555) $ &              & $ ( 969 , 5335) $ & $ 5169615 $ & $ ( 8 , 770077) $ & \\
$ ( 7 , 858) $ & $6006 $ & $ ( 707 , 858) $ & $ 606606 $   & $ ( 575 , 9119) $ & $ 5243425 $ & $ ( 77 , 80008) $ & \\

$ ( 77 , 88) $  & $6776  $ &  $ ( 7 , 88088) $ & $616616 $ & $ ( 979 , 5555) $ &                    & $ ( 88 , 70007) $ &  \\
$ ( 55 , 999) $ & $54945 $ &  $ ( 8 , 77077) $ &           & $ ( 55 , 99999) $ & $ 5499945 $        & $ ( 77 , 80088) $ & $\mathbf{6166776}$ \\ 
$ ( 99 , 555) $ &          &  $ ( 77 , 8008) $ &           & $ ( 99 , 55555) $ &                    & $ ( 88 , 70077) $ & \\
$ ( 77 , 858) $ & $66066 $ &  $ ( 88 , 7007) $ &           & $ ( 7 , 858088) $ & $\mathbf{6006616}$ & $ ( 898 , 7227) $ &  $ 6489846 $ \\
\hline
\end{tabular}
\end{table}

In addition to providing a positive answer to Problem~\ref{prob:1}, the table reveals some interesting facts. Except for the two values of 
$C$ printed in boldface, all other $C$s are themselves palindromes in which case {\em both} $A$ and $B$ are palindromes. The pairs 
$(7,858088)$, yielding $C=6006616$, and $(77,80088)$ and $(88,70077)$, giving $C=6166776$, contain a palindrome $A$. On the other hand, 
up to $C \leq 10^{7}$, no palindromic pairs were found, with neither $A$ nor $B$ being a palindrome, that was not polynomial. Computational 
evidence strongly suggests the following conjecture.

\begin{conjecture}\label{conj:main}
If $(A,B)$ is a palindromic pair, with neither $A$ nor $B$ a palindrome, then $(A,B)$ is a polynomial pair.
\end{conjecture} 
In attempting to answer the conjecture, we begin by establishing properties of polynomial pairs in the next section.

\section{Some Properties of Polynomial Pairs}\label{sec:three}
For $A \in \N$, let $A_{\infty}$ denote the maximum of the coefficients of $P(A,x)$.

\begin{proposition}\label{prop:ub}
If $(A,B)$ is a polynomial pair, then $A_{\infty} B_{\infty} \leq 9$. If, moreover, $A_\infty \geq 5$, then $B_{\infty}=1$.
\end{proposition}
\begin{proof}
This proposition is a direct consequence of Proposition~\ref{prop:equiv}. Let $j$ and $l$ be, respectively, the smallest index such that 
$a_{j}= A_{\infty}$ and $b_{l}=B_{\infty}$. Then $A_{\infty} B_{\infty} > 9$ would imply that the coefficient of $x^{j+l}$ in the multiplication 
$P(A,x) P(B,x)$ is $> 9$, violating Proposition~\ref{prop:equiv} Part 3).
\end{proof}
To derive a sufficient condition for $(A,B)$ to be a palindromic pair we define the norm of an integer by the formula $A_1=P(A,1).$
\begin{proposition}\label{prop:bound}
For $A,B \in \N$, $(AB)_{\infty} \leq A_{\infty} B_{1}$.
\end{proposition}
\begin{proof}
Write $P(A,x)=\sum_{i=0}^a a_i x^i,$ and $P(B,x)=\sum_{i=0}^b b_i x^i$. Then, the coefficient of $x^{k}$ in $P(AB,x)$ is
\begin{equation*}
\sum_{i+j=k} a_i b_j \le  A_\infty \sum_{j=0}^b b_j= A_\infty B_1 \text{.}
\end{equation*}
\end{proof}

A construction of polynomial pairs can be deduced from Proposition~\ref{prop:bound}.
\begin{proposition}\label{prop:const1}
Let $A,B \in \N$ with $ A_{\infty} B_{1} \le 9$. Then $(A,B)$ is a polynomial pair.
\end{proposition}
\begin{proof}
Combine Proposition~\ref{prop:equiv} Part 3) and Proposition~\ref{prop:bound}.
\end{proof}

Table~\ref{table:LNA} list downs all polynomial pairs $(A,B)$ with $A \times B = C$, $A \leq B$, and $C \leq C^{*} \leq 10^{4}$. Neither $A=A^{*}$ nor $B=B^{*}$ 
is allowed although $C = C^{*}$ is allowed. For each pair, there is a corresponding pair $(A^{*},B^{*})$ with $A^{*} \times B^{*} = C^{*}$. 
When $C$ is a palindrome, it is written in bold.
\begin{table}
\caption{The list of polynomial pairs $(A,B)$ with $A \leq B$,\newline
\hspace{\linewidth} neither $A=A^{*}$ nor $B=B^{*}$ is allowed, and $A \times B = C \leq C^{*} \leq 10^{4}$}
\label{table:LNA}
\centering
\begin{tabular}{ c  c  | c   c  | c   c | c   c | c  c }
\hline
$C$ & $(A,B)$ & $C$ & $(A,B)$ & $C$ & $(A,B)$ & $C$ & $(A,B)$ & $C$ & $(A,B)$ \\
\hline
$144 $ & $ ( 12 , 12 )$ & $1356 $ & $ ( 12 , 113 )$ & $2352 $ & $ ( 21 , 112 )$ &         & $ ( 24 , 112)$ & $3624 $ & $ ( 12 , 302)$ \\
$156 $ & $ ( 12 , 13 )$ & $1368 $ & $ ( 12 , 114 )$ & $2369 $ & $ ( 23 , 103 )$ & $2743 $ & $ ( 13 , 211)$ & $3648 $ & $ ( 12 , 304)$ \\
$168 $ & $ ( 12 , 14 )$ & $1428 $ & $ ( 14 , 102 )$ & $2373 $ & $ ( 21 , 113 )$ & $2769 $ & $ ( 13 , 213)$ & $3744 $ & $ ( 12 , 312)$ \\
$169 $ & $ ( 13 , 13 )$ & $1456 $ & $ ( 13 , 112 )$ & $2394 $ & $ ( 21 , 114 )$ & $\mathbf{2772}$ & $ ( 12 , 231)$ & $3768 $ & $ ( 12 , 314)$ \\
$\mathbf{252}$ & $ ( 12 , 21 )$ & $1464 $ & $ ( 12 , 122 )$ & $2436 $ & $ ( 12 , 203 )$ &         & $ ( 21 , 132)$ & $3864 $ & $ ( 12 , 322)$ \\

$273 $ & $ ( 13 , 21 )$ & $1469 $ & $ ( 13 , 113 )$ & $2448 $ & $ ( 12 , 204 )$ & $2793 $ & $ ( 21 , 133)$ & $3888 $ & $ ( 12 , 324)$ \\
$276 $ & $ ( 12 , 23 )$ & $1476 $ & $ ( 12 , 123 )$ &        & $ ( 24 , 102 )$ & $2796 $ & $ ( 12 , 233)$ & $3926 $ & $ ( 13 , 302)$ \\
$288 $ & $ ( 12 , 24 )$ & $1488 $ & $ ( 12 , 124 )$ & $2556 $ & $ ( 12 , 213 )$ & $2814 $ & $ ( 14 , 201)$ & $3984 $ & $ ( 12 , 332)$ \\
$294 $ & $ ( 14 , 21 )$ & $1568 $ & $ ( 14 , 112 )$ & $2562 $ & $ ( 21 , 122 )$ & $2873 $ & $ ( 13 , 221)$ & $4284 $ & $ ( 21 , 204)$ \\
$299 $ & $ ( 13 , 23 )$ & $1584 $ & $ ( 12 , 132 )$ & $2568 $ & $ ( 12 , 214)$ & $2892 $ & $ ( 12 , 241)$ &          & $ ( 42 , 102)$ \\

$384 $ & $ ( 12 , 32 )$ & $1586 $ & $ ( 13 , 122 )$ & $2576 $ & $ ( 23 , 112)$ & $2899 $ & $ ( 13 , 223)$ & $4386 $ & $ ( 43 , 102)$ \\
$1224 $ & $ ( 12 , 102 )$ & $1596 $ & $ ( 12 , 133 )$ & $2583 $ & $ ( 21 , 123)$ & $2954 $ & $ ( 14 , 211)$ & $4494 $ & $ ( 21 , 214)$ \\
$1236 $ & $ ( 12 , 103 )$ & $1599 $ & $ ( 13 , 123 )$ & $2599 $ & $ ( 23 , 113)$ & $3193 $ & $ ( 31 , 103)$ & $4669 $ & $ ( 23 , 203)$ \\
$1248 $ & $ ( 12 , 104 )$ & $2142 $ & $ ( 21 , 102 )$ & $2613 $ & $ ( 13 , 201)$ & $3264 $ & $ ( 32 , 102)$ & $4836 $ & $ ( 12 , 403)$ \\
$1326 $ & $ ( 13 , 102 )$ & $2163 $ & $ ( 21 , 103 )$ & $2639 $ & $ ( 13 , 203)$ & $3296 $ & $ ( 32 , 103)$ & $4899 $ & $ ( 23 , 213)$ \\

$1339 $ & $ ( 13 , 103 )$ & $2184 $ & $ ( 21 , 104 )$ & $2676 $ & $ ( 12 , 223)$ & $3468 $ & $ ( 34 , 102)$ & $4956 $ & $ ( 12 , 413)$ \\
$1344 $ & $ ( 12 , 112 )$ & $2346 $ & $ ( 23 , 102 )$ & $2688 $ & $ ( 12 , 224)$ & $3584 $ & $ ( 32 , 112)$ & $6496 $ & $ ( 32 , 203)$ \\
\hline
\end{tabular}
\end{table}

\section{Integers with Many Polynomial Pairs}\label{sec:four}
From Tables~\ref{table:NonPol} and~\ref{table:LNA} we notice that some $C \in \N$ can be the product of the elements of distinct polynomial pairs. 
Here we give a construction to show that some numbers can be the product of the elements of an arbitrarily large number of distinct polynomial pairs.

First, let us define a {\it repunit} $R(n)$ as the $n$-digit number whose digits are ones. The term, which abbreviates {\it repeated unit}, first appeared 
in~\cite[Ch.\ 11]{B}. More formally, 
\begin{equation*}
R(n)= \sum_{i=0}^{n-1} 10^{i} = \frac{10^{n}-1}{9}\text{.}
\end{equation*}	
Sequence $A004023$ in OEIS~\cite{OEIS} records the known values of $n$ for which $R(n)$ is prime.

\begin{theorem}
The repunit $R(2^{n})$ is the product of $n$ pairwise distinct positive integers. It can be expressed as the product $A \times B$ of
at least $M$ pairwise distinct polynomial pairs $(A,B)$ where $M$ is given by
\begin{equation}
M =
\begin{cases}
\sum_{j=1}^{k} {n \choose j}\text{,} & \text{ if } n=2k+1 \text{;}\\
\sum_{j=1}^{k-1} {n \choose j} + \frac{1}{2} {n \choose k}\text{,} & \text{ if } n=2k \text{.}
\end{cases}
\end{equation}
\end{theorem}

\begin{proof}
It is clear that $R(2)=10^{1}+1$ and $R(4)=R(2)(100+1)=(10^{1}+1)(10^{2}+1)$. Using the difference of squares, we can inductively write
\begin{align*}
R(2^{n}) & = \frac{1}{9} (10^{2^{n}}-1) = \frac{1}{9}(10^{2^{n-1}}-1)(10^{2^{n-1}}+1)\\
         & =R(2^{n-1}) (10^{2^{n-1}}+1) = \prod_{i=0}^{n-1} (10^{2^i}+1)\text{.}
\end{align*}
This establishes the first assertion. Moreover, all of the multiplications can be performed without carry 
since $R(a)_\infty =1$, for all integers $a\ge 1$.

Since there are $n$ distinct factors, we can group them into two disjoint nontrivial sets $\mathcal{A}$ and $\mathcal{B}$. 
Let $A$ be the product of the elements in $\mathcal{A}$ and $B$ analogously based on $\mathcal{B}$. Since we want to avoid repetition, 
two cases based on the parity of $n$ need to be considered.

When $n=2k+1$, the set $\mathcal{A}$ has $j$ elements with $1 \leq j \leq k=(n-1)/2$. The remaining $n-j$ elements not chosen 
for $\mathcal{A}$ automatically form the set $\mathcal{B}$. Thus we have $M=\sum_{j=1}^{k} {n \choose j}$.

When $n=2k$, we can do similarly for $1 \leq |\mathcal{A}| \leq k-1$ but we need to treat the case of $|\mathcal{A}|=|\mathcal{B}|=k$ 
with more care. To avoid forming repetitive pairs, we halve the count. In total, $M=\sum_{j=1}^{k-1} {n \choose j} + \frac{1}{2} {n \choose k}$. 
\end{proof}

\section{Reversal Multiplication}\label{sec:five}
A popular way to obtain palindromes is to multiply a number by its reversal.
This is called {\it reversal multiplication} in~\cite{won} and the numbers that give palindromes 
in that way form Sequence $A062936$ in OEIS~\cite{OEIS}. This recipe always works with polynomial pairs 
as the next result shows.

\begin{proposition}\label{prop:pal}
If $(A,A^{*})$ is a polynomial pair, then $A \times A^{*}$ is always a palindrome.
\end{proposition}
\begin{proof}
It suffices to confirm that $P(A^{*},x)=P^*(A,x)$ and that $P(A,x) P^*(A,x)$ is a self-reciprocal polynomial.
\end{proof}

It is observed in~\cite{won} that all elements $>3$ in Sequence $A062936$ only have digits $0,1$, and $2$. This is easy to show in the polynomial pair case and correlates with an observation made by David Wilson\footnote{https://oeis.org/A062936} on July 6, 2001 stating 
that said sequence includes positive integers not ending in $0$ whose sum of squares of the digits is $\leq 9$.

\begin{proposition}\label{prop:CS}
If $(A,A^{*})$ is a polynomial pair, then the sum of the squares of the digits of $A$ is $\leq 9$. 
In particular, if $A>9$, we have  $A_\infty \leq 2$. Conversely, if the sum of the squares of the digits of $A$ is $\leq 9$, then $(A,A^{*})$ is a polynomial pair.
\end{proposition}
\begin{proof}
Write $P(A,x)=\sum_{i=0}^{d} a_{i} x^{i}$. By Cauchy-Schwarz inequality, for $0 \leq k \leq d$,  
\begin{equation*} 
\sum_{l=0}^{k} a_{l} a_{d-k+l} \leq \sum _{i=0}^{d} a_{i}^{2}\text{.}
\end{equation*}
The left hand side is the coefficient of $x^{k}$ while the right hand side is the coefficient of $x^{d}$ in $P(A \times A^{*},x)$. 
Thus, by Proposition~\ref{prop:pal}, 
\begin{equation}\label{eq:ineq}
(A \times A^{*})_{\infty} = \sum_{i=0}^{d} a_{i}^{2} \text{,}
\end{equation}
which is the sum of the squares of the digits of $A$. This establishes the first statement from which follows that if $A$ has at least two nonzero digits, then none can be $ \geq 3$.

The converse follows from Equation (\ref{eq:ineq}) and Proposition~\ref{prop:equiv}.
\end{proof}

We have generated a list of elements $A <10^{9}$ of Sequence $A062936$ and verified that the sum of the squares of the digits of $A$ is bounded above by $9$. Applying the converse part of Proposition~\ref{prop:CS}, we are led to the following conjecture.
\begin{conjecture}\label{2}
If $A \times A^{*}$ is a palindrome, then $(A,A^{*})$ is a polynomial pair.
\end{conjecture}

To emphasize that we make Conjecture~\ref{2} for base $b=10$ only, we observe the following counterexamples for $b \neq 10$. 

In base $2$, we have $11 \times 11= 1001$. More generally, for any integer $l \geq 2$,
\begin{multline}\label{eq:counter}
11 \underbrace{00\cdots0}_{2l} 10101 \underbrace{00\cdots0}_{2l-1} 11 \times 11 \underbrace{00\cdots0}_{2l-1} 10101 \underbrace{00\cdots0}_{2l} 11\\
=  1001 \times 2^{8l+12}+ 10111101 \times 2^{6l+7} + 1001001001 \times 2^{4l+3 }+10111101 \times 2^{2l+1}  + 1001 \text{,}
\end{multline}
which can be seen to be a palindrome. Using $2l+1$, instead of $2l-1$, also works. Our computation reveals that there are no other counterexamples 
with $A$ having less than $20$-digit base $2$ representation.

In base $4$, the only counterexample with $A$ having less than 10-digit representation is $2232213 \times 3122322 = 21111033011112$. The next counter 
example, if exists, must be a considerably large number. 

Table~\ref{table:LAAstar} gives the counterexamples we found for base $b \in \{3,4,5,7,8,9,11\}$.  We have not been able to find 
counterexamples in either base $6$ or base $10$.
\begin{table}[h!t!]
\caption{Examples of $A$ and $A\times A^*$ where $(A,A^{*})$ is not a polynomial pair\newline
\hspace{\linewidth}for $A \le A^*$ in bases $b \in \{3,4,5,7,8,9,11\}$.\newline
\hspace{\linewidth}In base $11$, $a$ stands for $10$}
\label{table:LAAstar}
\centering
\begin{tabular}{ c   c   c  | c   c   c  }
\hline
Base & $A$ & $A \times A^{*}$ & Base & $A$ & $A \times A^{*}$  \\
\hline
$3$ & $ 2 $          & $ 11 $                  & $8$ & 	$ 47 $   & $ 4444 $ \\
    & $ 202 $        & $ 112211 $              &     &	$ 303 $  & $ 112211 $ \\
    & $ 2002 $       & $ 11022011 $            &     & 	$ 306 $  & $ 225522 $ \\
    & $ 20002 $      & $ 1100220011 $          &     & 	$ 333 $  & $ 135531 $ \\
    & $ 200002 $     & $ 110002200011 $        &     & 	$ 3003 $ & $ 11022011 $ \\

    & $ 201102 $     & $ 111221122111 $        &     &	$ 3006 $ & $ 22055022 $ \\
    & $ 2000002 $    & $ 11000022000011 $      &     & 	$ 3033 $ & $ 12244221 $ \\
    & $ 20000002 $   & $ 1100000220000011 $    &     & 	$ 3116 $ & $ 23300332 $ \\
    & $ 20011002 $   & $ 1101211111121011 $    &     & 	$ 3306 $ & $ 24377342 $ \\
    & $ 200000002 $  & $ 110000002200000011 $  &     & 	$ 3333 $ & $ 13577531 $ \\

    & $ 2000000002 $ & $ 11000000022000000011 $ & 	 &  $ 30003 $ & $ 1100220011 $ \\
    & $ 2000110002 $ & $ 11001210111101210011 $ & 	 &  $ 30006 $ & $ 2200550022 $ \\
$4$ & $ 2232213 $    & $ 21111033011112 $       &    &  $ 30033 $ & $ 1211441121 $ \\
$5$ & $ 314 $        & $ 242242 $               &    &  $ 30303 $ & $ 1122332211 $ \\
    & $ 22033 $      & $ 1334334331 $          & $9$ &  $ 516 $   & $ 350053 $ \\

    & $ 220033 $     & $ 133133331331 $         &     & $ 44055 $   & $ 2667557662 $ \\							
    & $ 2200033 $    & $ 13310333301331 $       &     & $ 440055 $  & $ 266255552662 $ \\
    & $ 2301123 $    & $ 14000211200041 $       &     & $ 2403555 $ & $ 14682744728641 $ \\
    & $ 22000033 $   & $ 1331003333001331 $     &     & $ 4400055 $ & $ 26620555502662 $ \\
    & $ 23410123 $   & $ 1423123003213241 $    & $11$ & $6$         & $ 33 $ \\

    & $ 24200303 $   & $ 1344343003434431 $     &  & $ 66 $  & $ 3993 $ \\							
$7$ & $ 4 $          & $ 22 $                   &  & $ 77 $  & $ 5335 $ \\
    & $ 44 $         & $ 2662 $                 &  & $ 374 $ & $ 161161 $ \\
    & $ 55 $         & $ 4444 $                 &  & $ 419 $ & $ 350053 $ \\
    & $ 404 $        & $ 224422 $               &  & $ 606 $ & $ 336633 $ \\

    & $ 4004 $       & $ 22044022 $             &  	& $ 6006 $  & $ 33066033  $ \\
    & $ 4114 $       & $ 23300332$              &  	& $ 21896 $ & $ 139a00a931 $ \\
    & $ 25124 $      & $ 1456446541 $           &  	& $ 33088 $ & $ 1669999662 $ \\
    & $ 40004 $      & $ 2200440022 $           &  	& $ 60006 $ & $ 3300660033 $ \\
    & $ 400404 $     & $ 222426624222 $         &  	& $ 60606 $ & $ 3366996633 $ \\

    & $ 403304 $     & $ 223652256322 $         & 	& $ 63328 $  & $ 4779559774 $ \\
    & $ 404004 $     & $ 222426624222 $         & 	& $ 283306 $ & $ 156695596651 $ \\
    & $ 4000004 $    & $ 22000044000022 $       &  	& $ 330088 $ & $ 266279972662 $ \\
$8$ & $ 3 $          & $ 11 $                   &  	& $ 391744 $ & $ 15a484484a51 $ \\
    & $ 6 $          & $ 44 $                   &  	& $ 441739 $ & $ 379373373973 $ \\

    & $ 33 $         & $ 1331 $                 &   & $ 600006 $ & $ 330006600033 $ \\
    & $ 36 $         & $ 2772 $                 &   & $ 600606 $ & $ 333639936333 $ \\
\hline
\end{tabular}
\end{table}

\begin{proposition}\label{prop:imp}
There are infinitely many values of $b$ for which the analogue of Conjecture~\ref{2} in base $b$ is false.
\end{proposition}
\begin{proof}
First, consider the base $b$ such that $b=r^{2}-1$ for $2 \le r \in \N$. Writing in base $b$,
$r \times r = 11$, and for any non-negative integer $j$,
\begin{align*}
r \underbrace{00\cdots0}_{j+1} r \times r \underbrace{00\cdots0}_{j+1} r & = (r \times b^{j+2} + r)^{2} = (r^{2} \times b^{2j+4}) + (2 \times r^{2} \times b^{j+2}) + r^{2} \\
&= 11 \underbrace{00\cdots0}_{j}22 \underbrace{00\cdots0}_{j} 11 \text{.}
\end{align*}
There are obviously infinitely many such bases $b$.

For any base $b$ of the form $b = 4k- 1$, 
\begin{equation}\label{eq:two}
(2k) \times (2k) = 4k^{2}= (k \times b) + k = kk
\end{equation}
in base $b$. More generally, in the said base, one can easily verify, by using Equation~\ref{eq:two}, that 
\begin{align*}
((2k) \underbrace{00\cdots0}_{j+1} (2k))^{2} = ((2k \times b^{j+2}) + 2k)^{2} &= (4k^{2} \times b^{2j+4}) + (2 \times 4k^{2} \times b^{j+2}) + 4k^{2} \\
                                             &= kk \underbrace{00\cdots0}_{j} (2k) (2k) \underbrace{00\cdots0}_{j} kk \text{.}
\end{align*}

When the base $b$ is of the form $b=4k+1$, we can write
\begin{equation}\label{eq:twoplus}
(2k) \times (2k+1) = k \times (4k+1) + k = (k \times b) + k = kk \text{.}
\end{equation}
Using Equation~\ref{eq:twoplus}, one gets
\begin{align*} 
& [(2k)(2k) 00 (2k+1)(2k+1)] \times [(2k+1)(2k+1)00(2k)(2k)] \\ 
&= [ b^{4} \times 2k \times (b+1) +  (2k+1) \times (b+1) ] \times [ b^{4} \times (2k+1) \times (b+1) + 2k \times (b + 1) ] \\
&= kk + [(2k)(2k) \times b] + [kk \times b^{2}] + [(2k)(2k+1) \times b^{4}] + [101 \times b^{5}] + [(2k)(2k+1) \times b^{6}] \\
& + [kk \times b^{8}] + [(2k)(2k) \times b^{9}] + [kk \times b^{10}] \\
&= k(3k)(3k)k (2k+1)(2k+1)(2k+1)(2k+1) k(3k)(3k)k \text{.}
\end{align*}
In fact, one can obtain a slightly more general result since, for any non-negative integer $j$,
\begin{multline*} 
(2k)(2k) \underbrace{00\cdots0}_{j+2} (2k+1)(2k+1) \times (2k+1)(2k+1)\underbrace{00\cdots0}_{j+2}(2k)(2k) \\ 
= k(3k)(3k)k\underbrace{00\cdots0}_{j} (2k+1)(2k+1)(2k+1)(2k+1) \underbrace{00\cdots0}_{j} k(3k)(3k)k \text{.} 
\end{multline*}
\end{proof}

Thus, a necessary but insufficient condition for the analogue of Conjecture~\ref{2} in base $b$ to hold is for $b$ to be even and for $b+1$ to be square-free.

\begin{remark}
Let $(A,B)$ be a palindromic pair. If either $A$ or $B$ is itself a palindrome, then we cannot conclude immediately that $(A,B)$ is a polynomial pair. Indeed, in many cases, for example, when $A=121$ and $B=A^{*}=A$, the pair $(A,B)$ is both palindromic and polynomial. Yet, as shown by the pairs listed in Table~\ref{table:NonPol}, a palindromic pair may fail to be polynomial when either $A$ or $B$ is a palindrome.

Conjecture~\ref{2} posits that, regardless of whether $A$ itself is a palindrome, so long as $A \times A^{*}$ is a palindrome, then $(A,A^{*})$ is polynomial. Thus, this conjecture does not follow from Conjecture~\ref{conj:main}. If, however, we add the condition that $A \neq A^{*}$, then a positive answer to Conjecture~\ref{conj:main} settles this modified version of Conjecture~\ref{2} since, if $A \times A^{*}$ is a palindrome, then $(A,A^{*})$ is of course a palindromic pair.

Note that Proposition~\ref{prop:imp} still holds if we use the base $b$ analogue for the modified version of Conjecture~\ref{2} using only bases $b=4k+1$ in the proof. In this case, removing all entries in Table~\ref{table:LAAstar} having $A=A^{*}$ provides analogous examples.
\end{remark}

To end this section we prove a special case of Conjecture~\ref{2}.
\begin{proposition}
If $A$ is an $n-$digit number and $A \times A^{*}$ is a $(2n-1)$-digit palindrome then $(A,A^{*})$ is a polynomial pair.
\end{proposition}

\begin{proof}
Let $A$ be an $n$-digit number such that $A\times A^*$ is a $(2n-1)$-digits palindrome, with the notation $P(A,x) = \sum_{i=0}^{n-1} a_{i}x^{i}$. 
Let $c_{0},c_{1}, \ldots, c_{2n-2}$ be the digits of $A \times A^*$. We now make completely explicit how the digits are manipulated 
when the multiplication is performed. 

Let $\gamma_i$ be the carry that is propagated on the $i$-th digits and $\sigma_i$ be the sum of the products of 
digits that appear in the $i$-th position. Hence,
$$ \gamma_0= 0 $$
and, for all $0 \le i \le 2n-1$,
\begin{align*}
\sigma_i &= \displaystyle \gamma_i + \sum_{k=\max(0,i+1-n)}^{\min (n-1,i)} a_k \, a_{n-1-i+k}\text{,}\\
c_i &= \sigma_i \pmod{10} \text{,}\\
\gamma_{i+1} &= (\sigma_i-c_i)/10 \text{.}
\end{align*}

Note that $(A,A^*)$ is a polynomial pair if and only if $\gamma_{i}=0$ for all $i\le 2n-1$. We prove this fact by induction. 

Since $A \times A^*$ has only $2n-1$ digits, we have $c_{2n-1} =0$, and thus $\gamma_{2n-1} =0$. 
Suppose that for a certain integer $\ell$ we have proven that $\gamma_{\ell} =0$ and $\gamma_{2n-\ell-1} =0$. 
Since $\gamma_{2n-\ell-1}=0$, we must have 
\begin{equation*}
\sigma_{2n-\ell-2} = c_{2n-\ell-2}\le 9 \text{.}
\end{equation*}
Now, 
\begin{equation*}
\sigma_\ell = \sigma_{2n-2-\ell} - \gamma_{2n-\ell-2} + \gamma_\ell = \sigma_{2n-\ell-2} - \gamma_{2n-\ell-2}
\end{equation*}
must be $\leq 9$ too. 
So $\gamma_{\ell +1} = 0$ and $c_{\ell} = \sigma_{\ell}$. Since $A \times A^*$ is a palindrome, we have $c_\ell = c_{2n-\ell-2}$. 
So we also have $ \sigma_{2n-\ell-2} = \sigma_\ell$. Now we can compute that 
\begin{equation*}
\gamma_{2n-\ell-2} = \sigma_{2n-\ell-2} - \sigma_\ell + \gamma_\ell = 0\text{,}
\end{equation*}
which concludes the induction step.
\end{proof}

\section{Squares and Palindromes}\label{sec:six}
In this short section we show that some results established above shed light on several connections between palindromes and squares. 

There are two sequences in OEIS~\cite{OEIS} concerning palindromes and squares. Sequence $A002779$ lists down palindromic perfect squares while Sequence $A002778$ contains integers whose squares are palindromes. The next result, which is a direct consequence of Proposition~\ref{prop:pal}, gives a sufficient but not a necessary condition for an integer $A$ to belong to Sequence $A002778$.
\begin{proposition}
If $(A,A)$ is a polynomial pair with $A$ a palindrome, then $A^{2}$ is a palindrome.
\end{proposition}

Each entry of Sequence $A156317$ in OEIS~\cite{OEIS} is a perfect square that forms either an equal or a larger perfect square when reversed. Here is a technique to produce examples of such integers.
\begin{proposition}
If $(A,A)$ is a polynomial pair then so is $(A^{*},A^{*})$. Moreover, $(A^{2})^{*} = (A^{*})^{2}$.
\end{proposition}

\begin{proof} 
It suffices to verify that $P((A^{2})^{*},x) = P^{*}(A^{2},x) = (P^{*}(A,x))^{2} = (P(A^{*},x))^{2}$.
\end{proof}

\section{Additive Pairs}\label{sec:seven}
It is natural to consider as well the additive analogue of polynomial pairs. The pair $(A,B)$ of positive integers is said to be 
an {\it additive pair} if
\begin{equation*}
P(A,x)+P(B,x)=P(A+B,x) \text{.}
\end{equation*}
The counterpart of Proposition~\ref{prop:equiv} can then be established.
\begin{proposition}\label{prop:equiv2}
The following assertions are equivalent.
\begin{enumerate}[1)]
\item The pair $(A,B)$ is an additive pair.
\item The addition of $A$ by $B$ can be performed without carry.
\item The coefficients of the polynomial $P(A,x)+P(B,x)$ are bounded above by $9$.
\end{enumerate}
\end{proposition}
\begin{proof}
We use the same representation of $P(A,x)$ and $P(B,x)$ as in the proof of Proposition~\ref{prop:bound}. 
Let $j$ be  the smallest integer such that $c_j = a_j + b_j >9$. Then $c_j$ is the coefficient of $x^{j}$ in $P(A,x)+P(B,x)$ 
while $c_j \pmod{10} \neq c_j$ is the coefficient of $x^j$ in $P(A+B,x)$. By contrapositive argument, 1) implies 2).

It is clear by definition of polynomial addition that 2) implies 3). To verify that 3) implies 1) note that for $0 \leq j \leq \max(a,b)$ 
we have $c_j=a_j+b_j \leq 9$, which leads immediately to the desired conclusion since $c_j$ is the coefficient of $x^{j}$ in both $P(A+B,x)$ 
and $P(A,x)+P(B,x)$.
\end{proof}

A sufficient condition for $(A,B)$ to be an additive pair is $A_\infty + B_\infty \leq 9.$ Additive pairs can be used to generates palindromes.

\begin{proposition}
If $(A,A^{*})$ is an additive pair, then $A + A^{*}$ is a palindrome.
\end{proposition}

\begin{proof}
It is straightforward to verify that $P(A^{*},x)=P^*(A,x)$ and that $P(A+A^{*},x)=P(A,x)+P^{*}(A,x)$ is a self-reciprocal polynomial.
\end{proof}

There are, however, integers $A$ such that $A+A^{*}$ is a palindrome yet $(A,A^{*})$ is not an additive pair. The numbers $56$ and $506$ are some easy examples of such $A$.

\section{Summary}
In this note we have shown how to use polynomial pairs to study the properties of palindromic pairs. Furthermore, a large number of palindromic pairs can be constructed by using polynomial pairs. Connections to well-known numbers and integer sequences in OEIS have also been explicated.

It is of interest to either find counterexamples to or to prove the validity of the conjectures mentioned here for 
future investigations. As an added incentive, we offer a ripe durian for a correct proof of, or a valid counterexample to, any of the conjectures. 
 
\section{Acknowledgement}
The authors thank Abdul Qatawneh for helpful discussions.

\bigskip
\hrule
\bigskip
\noindent 2010 {\it Mathematics Subject Classification}:
Primary 11B75; Secondary 97A20.

\noindent \emph{Keywords: } 
number reversal, palindromes, palindromic pair, polynomial pair, repunit 

\bigskip
\hrule
\bigskip

\noindent (Concerned with sequences \seqnum{A002778}, \seqnum{A004023}, \seqnum{A062936}, \seqnum{A156317})

\bigskip
\hrule
\bigskip

\vspace*{+.1in}
\noindent
Received {\it to be supplied}.
revised version received  {\it to be supplied}
Published in {\it Journal of Integer Sequences}, {\it to be supplied}.

\bigskip
\hrule
\bigskip

\noindent
Return to
\htmladdnormallink{Journal of Integer Sequences home page}{http://www.cs.uwaterloo.ca/journals/JIS/}.
\vskip .1in


\begin{thebibliography}{99}
\bibitem{BC}
W.~W.~Rouse Ball and H.~S.~M.~Coxeter, 
{\it Mathematical Recreations and Essays}, Dover, 2007.

\bibitem{B}
A.~H.~Beiler, 
{\it Recreations in the Theory of Numbers: The Queen of Mathematics Entertains}, Dover, 1966. 

\bibitem{OEIS}
N.~J.~A.~Sloane, The online encyclopedia of integer sequences, \href{http://oeis.org/}{\tt http://oeis.org.}

\bibitem{won}
P.~de Geest, Palindromic products of integers and their reversals, \\
\href{www.worldofnumbers.com/reversal.htm}{\tt www.worldofnumbers.com/reversal.htm}.

\end{thebibliography}
\end{document}